\documentclass[12pt, reqno]{amsart}
\usepackage[marginratio=1:1,height= 620pt, width=450pt,tmargin =80pt]
{geometry}
\usepackage{amsmath}
\usepackage{amsfonts}
\usepackage{amssymb}
\usepackage[all]{xy}           
\usepackage{bm}
\usepackage{bbding}
\usepackage{txfonts}
\usepackage{amscd}
\usepackage[english]{babel}
\usepackage{xspace}
\usepackage[shortlabels]{enumitem}
\usepackage{mathrsfs}
\usepackage{ifpdf}

\usepackage[colorlinks=true, final, linkcolor=red,  citecolor=green, backref, hyperindex]{hyperref}
\usepackage{tikz}
\usepackage[active]{srcltx}
\usepackage{xcolor}
\definecolor{rose}{rgb}{0,0.4,0.6}
\usepackage{xcolor}

\makeatletter

\newtheorem{theorem}{Theorem}[section]
\newtheorem{prop}[theorem]{Proposition}
\newtheorem{thm}[theorem]{Theorem}
\newtheorem{lemma}[theorem]{Lemma}

\newtheorem{prop-def}{Proposition-Definition}[section]

\theoremstyle{definition}
\newtheorem{defn}[theorem]{Definition}

\newtheorem{remark}[theorem]{Remark}
\newtheorem{exam}[theorem]{Example}

\newcommand{\nc}{\newcommand}

\nc{\delete}[1]{{}}
\nc{\mmargin}[1]{}
\nc{\Alg}{\mathrm{Alg}}
\nc{\AvO}{\mathrm{AvO}}
\nc{\AvA}{\mathrm{AvA}}
\nc{\rmH}{\mathrm{H}}
\nc{\mlabel}[1]{\label{#1}}  
\nc{\mcite}[1]{\cite{#1}}  
\nc{\mref}[1]{\ref{#1}}  
\nc{\mbibitem}[1]{\bibitem{#1}} 

\delete{
	\nc{\mlabel}[1]{\label{#1}  
		{\hfill \hspace{1cm}{\bf{{\ }\hfill(#1)}}}}
	\nc{\mcite}[1]{\cite{#1}{{\bf{{\ }(#1)}}}}  
	\nc{\mref}[1]{\ref{#1}{{\bf{{\ }(#1)}}}}  
	\nc{\mbibitem}[1]{\bibitem[\bf #1]{#1}} 
}

 \font\cyrs=wncyr7

\newcommand{\bk}{{\mathbf{k}}}

\nc{\bincc}[2]{  ( {\scs{#1} \atop
		\vspace{-1cm}\scs{#2}} )}  
\nc{\oline}[1]{\overline{#1}}
\nc{\mapm}[1]{\lfloor\!|{#1}|\!\rfloor}
\nc{\bs}{\bar{S}}
\nc{\la}{\longrightarrow}
\nc{\ot}{\otimes}
\nc{\rar}{\rightarrow}
\nc{\lon }{\,\rightarrow\,}
\nc{\dar}{\downarrow}
\nc{\dap}[1]{\downarrow \rlap{$\scriptstyle{#1}$}}
\nc{\defeq}{\stackrel{\rm def}{=}}
\nc{\dis}[1]{\displaystyle{#1}}
\nc{\dotcup}{\ \displaystyle{\bigcup^\bullet}\ }
\nc{\hcm}{\ \hat{,}\ }
\nc{\hts}{\hat{\otimes}}
\nc{\hcirc}{\hat{\circ}}
\nc{\lleft}{[}
\nc{\lright}{]}
\nc{\curlyl}{\left \{ \begin{array}{c} {} \\ {} \end{array}
	\right .  \!\!\!\!\!\!\!}
\nc{\curlyr}{ \!\!\!\!\!\!\!
	\left . \begin{array}{c} {} \\ {} \end{array}
	\right \} }
\nc{\longmid}{\left | \begin{array}{c} {} \\ {} \end{array}
	\right . \!\!\!\!\!\!\!}
\nc{\ora}[1]{\stackrel{#1}{\rar}}
\nc{\ola}[1]{\stackrel{#1}{\la}}
\nc{\scs}[1]{\scriptstyle{#1}} \nc{\mrm}[1]{{\rm #1}}
\nc{\dirlim}{\displaystyle{\lim_{\longrightarrow}}\,}
\nc{\invlim}{\displaystyle{\lim_{\longleftarrow}}\,}
\nc{\dislim}[1]{\displaystyle{\lim_{#1}}} \nc{\colim}{\mrm{colim}}
\nc{\mvp}{\vspace{0.3cm}} \nc{\tk}{^{(k)}} \nc{\tp}{^\prime}
\nc{\ttp}{^{\prime\prime}} \nc{\svp}{\vspace{2cm}}
\nc{\vp}{\vspace{8cm}}
\nc{\modg}[1]{\!<\!\!{#1}\!\!>}
\nc{\intg}[1]{F_C(#1)}
\nc{\lmodg}{\!<\!\!}
\nc{\rmodg}{\!\!>\!}
\nc{\cpi}{\widehat{\Pi}}
\nc{\ssha}{{\mbox{\cyrs X}}} 
\nc{\tsha}{{\mbox{\cyrt X}}}
\nc{\shpr}{\diamond}    
\nc{\labs}{\mid\!}
\nc{\rabs}{\!\mid}

\nc{\ad}{\mrm{ad}}
\nc{\ann}{\mrm{ann}}
\nc{\Aut}{\mrm{Aut}}
\nc{\bim}{\mbox{-}\mathsf{Bimod}}
\nc{\br}{\mrm{bre}}
\nc{\can}{\mrm{can}}
\nc{\Cont}{\mrm{Cont}}
\nc{\rchar}{\mrm{char}}
\nc{\cok}{\mrm{coker}}
\nc{\de}{\mrm{dep}}
\nc{\dtf}{{R-{\rm tf}}}
\nc{\dtor}{{R-{\rm tor}}}

\nc{\Div}{{\mrm Div}}
\nc{\Diff}{\mrm{DA}}
\nc{\Diffl}{\mathsf{DA}_\lambda}
\nc{\diffo}{{\mathsf{DO}_\lambda}}
\nc{\alg}{\mathsf{Alg}}
\nc{\End}{\mrm{End}}
\nc{\Ext}{\mrm{Ext}}
\nc{\Fil}{\mrm{Fil}}
\nc{\Fr}{\mrm{Fr}}
\nc{\Frob}{\mrm{Frob}}
\nc{\Gal}{\mrm{Gal}}
\nc{\GL}{\mrm{GL}}
\nc{\Hom}{\mrm{Hom}}
\nc{\Hoch}{\mrm{Hoch}}
\nc{\hsr}{\mrm{H}}
\nc{\hpol}{\mrm{HP}}
\nc{\id}{\mrm{id}}
\nc{\im}{\mrm{im}}
\nc{\Id}{\mrm{Id}}
\nc{\ID}{\mrm{ID}}
\nc{\Irr}{\mrm{Irr}}
\nc{\incl}{\mrm{incl}}
\nc{\length}{\mrm{length}}
\nc{\NLSW}{\mrm{NLSW}}
\nc{\Lie}{\mrm{Lie}}
\nc{\mchar}{\rm char}
\nc{\mpart}{\mrm{part}}
\nc{\ql}{{\QQ_\ell}}
\nc{\qp}{{\QQ_p}}
\nc{\rank}{\mrm{rank}}
\nc{\rcot}{\mrm{cot}}
\nc{\rdef}{\mrm{def}}
\nc{\rdiv}{{\rm div}}
\nc{\rtf}{{\rm tf}}
\nc{\rtor}{{\rm tor}}
\nc{\res}{\mrm{res}}
\nc{\SL}{\mrm{SL}}
\nc{\Spec}{\mrm{Spec}}
\nc{\tor}{\mrm{tor}}
\nc{\Tr}{\mrm{Tr}}
\nc{\tr}{\mrm{tr}}
\nc{\wt}{\mrm{wt}}
\def\ot{\otimes}

\nc{\bfk}{{\bf k}}
\nc{\bfone}{{\bf 1}}
\nc{\bfzero}{{\bf 0}}
\nc{\detail}{\marginpar{\bf More detail}
	\noindent{\bf Need more detail!}
	\svp}
\nc{\gap}{\marginpar{\bf Incomplete}\noindent{\bf Incomplete!!}
	\svp}
\nc{\FMod}{\mathbf{FMod}}
\nc{\Int}{\mathbf{Int}}
\nc{\Mon}{\mathbf{Mon}}
\nc{\remarks}{\noindent{\bf Remarks: }}
\nc{\Rep}{\mathbf{Rep}}
\nc{\Rings}{\mathbf{Rings}}
\nc{\Sets}{\mathbf{Sets}}
\nc{\ob}{\mathsf{Ob}}

\nc{\BA}{{\mathbb A}}   \nc{\CC}{{\mathbb C}}
\nc{\DD}{{\mathbb D}}   \nc{\EE}{{\mathbb E}}
\nc{\FF}{{\mathbb F}}   \nc{\GG}{{\mathbb G}}
\nc{\HH}{{\mathbb H}}   \nc{\LL}{{\mathbb L}}
\nc{\NN}{{\mathbb N}}   \nc{\PP}{{\mathbb P}}
\nc{\QQ}{{\mathbb Q}}   \nc{\RR}{{\mathbb R}}
\nc{\TT}{{\mathbb T}}   \nc{\VV}{{\mathbb V}}
\nc{\ZZ}{{\mathbb Z}}   \nc{\TP}{\widetilde{P}}
\nc{\m}{{\mathbbm m}}

\nc{\cala}{{\mathcal A}}    \nc{\calc}{{\mathcal C}}
\nc{\cald}{\mathcal{D}}     \nc{\cale}{{\mathcal E}}
\nc{\calf}{{\mathcal F}}    \nc{\calg}{{\mathcal G}}
\nc{\calh}{{\mathcal H}}    \nc{\cali}{{\mathcal I}}
\nc{\call}{{\mathcal L}}    \nc{\calm}{{\mathcal M}}
\nc{\caln}{{\mathcal N}}    \nc{\calo}{{\mathcal O}}
\nc{\calp}{{\mathcal P}}    \nc{\calr}{{\mathcal R}}
\nc{\cals}{{\mathcal S}}    \nc{\calt}{{\Omega}}
\nc{\calv}{{\mathcal V}}    \nc{\calw}{{\mathcal W}}
\nc{\calx}{{\mathcal X}}

\nc{\fraka}{{\mathfrak a}}
\nc{\frakb}{\mathfrak{b}}
\nc{\frakg}{{\frak g}}
\nc{\frakl}{{\frak l}}
\nc{\fraks}{{\frak s}}
\nc{\frakB}{{\frak B}}
\nc{\frakm}{{\frak m}}
\nc{\frakM}{{\frak M}}
\nc{\frakp}{{\frak p}}
\nc{\frakW}{{\frak W}}
\nc{\frakX}{{\frak X}}
\nc{\frakS}{{\frak S}}
\nc{\frakA}{{\frak A}}
\nc{\frakx}{{\frakx}}

\nc{\lir}[1]{\textcolor{red}{\underline{Li:}#1 }}

\begin{document}

\title[Deformations theory and minimal model of operads for Nijenhuis algebras morphisms]
{Deformations theory and minimal models of operads for Nijenhuis algebras morphisms}

\author{Sami Benabdelhafidh}
\address{University of Sfax, Faculty of Sciences of Sfax, BP 1171, 3038 Sfax, Tunisia.}
\email{\bf abdelhafidhsami41@gmail.com}

\begin{abstract} 
Nijenhuis operators are very useful in the deformation theory of algebras. 
In this paper, we introduce a new cohomology theory related to deformation of Nijenhuis algebra morphisms, this notion involves simultaneous deformation of two Nijenhuis algebras and a morphism between them.
As a consequence, we define a cohomology theory of Nijenhuis algebra morphisms to interpret the lower degree cohomology groups as formal deformation. We also prove a cohomology comparison Theorem of Nijenhuis algebra morphisms, i.e. the cohomology of a morphism of Nijenhuis algebras is isomorphic to the cohomology of an auxiliary Nijenhuis algebra. Finally, we construct a minimal model for the operad governing Nijenhuis algebras morphisms.

\end{abstract}

\subjclass[2020]{
16D20, 
16E40, 
16E40,   
16S80.  
}

\keywords{ Associative algebra, Nijenhuis operator, morphism, cohomology, deformation, ($\mathrm{CCT}$) Theorem}

\maketitle

\tableofcontents

\allowdisplaybreaks

\section{Introduction}
An interesting operator arising in the study of linear deformation theory of algebraic structures \cite{Gerst3}, integrable systems and tensor hierarchies in mathematical physics \cite{koss}, quantum bi-Hamiltonian systems \cite{Cari} called a Nijenhuis operator. 
In recent years, due to the outstanding work of 
\cite{azimi,baishya-das2,ebrahimi,ebrahimi-leroux,lei,leroux,liu,liu-sheng,ma,saha,peng,wang,yuan}, more and more scholars
begun to pay attention to the Nijenhuis operators.
Leroux \cite{leroux} introduced the notion of an NS-algebra as the algebraic structure induced by a Nijenhuis operator (see also \cite{lei}). NS-algebras are a generalization of dendriform algebras and are closely related to twisted Rota-Baxter operators \cite{uchino}.
Ebrahimi-Fard \cite{ebrahimi} interpreted the associative analogue
of the Nijenhuis relation as the homogeneous version of the Rota-Baxter relation.
Das \cite{das-ns} has introduced the cohomology and deformations theory of an NS-algebra using the idea of nonsymmetric operads with multiplication. 

A deformation of a mathematical object, roughly speaking, means that
it preserves its original structure after a parameter perturbation. In physics, deformation theory
discovers from quantizing classical mechanics, and this idea promotes some researches on quantum
groups in mathematics \cite{VC01,RH1}. Recently, deformation quantization has produced many elegant
works in the context of mathematical physics. Based on work of complex analysis by Kodaira and
Spencer \cite{KK1}, deformations theory was developed in algebra \cite{RH1}. 
In 1983, Gerstenhaber and Schack~\cite{Gerst1,Gerst2} introduced a deformation theory for morphisms of associative algebras and develop a powerful result called the
cohomology comparison Theorem ($\mathrm{CCT}$). 
Fr\'{e}gier~\cite{Fr} extended this idea to Lie algebra morphisms. Morphisms cohomology and deformations of Hom-algebras were established in \cite{Mak}.  
Du and Bao~\cite{Bao} examined the cohomology and deformation of Rota-Baxter operator on associative algebra morphisms. 
Unlike Rota-Baxter operators, it is important to observe that Nijenhuis operators cannot be characterized by their graphs as subalgebras of some bigger algebras. This makes the study of Nijenhuis operators a little more complicated and thus some results about Nijenhuis operators are not analogous to the study of Rota-Baxter operators. 
This paper contributes to addressing these difficulties in the context of
Nijenhuis structures.

Another significant challenge in the study of algebraic structures is understanding their homotopy analogs
just like $A_\infty$-algebras for usual associative algebras. From the perspective of operad theory,
specifically, the task is to formulate a cofibrant resolution for the operad of an algebraic structure.
The most desirable outcome would be providing a minimal model of the operad governing the
algebraic structure. When this operad is Koszul, there exists a general theory, the so-called Koszul
duality for operads \cite{GJ94,GK94}, which defines a homotopy version of this algebraic structure via
the cobar construction of the Koszul dual cooperad, which, in this case, is a minimal model. However, for
operads that are not Koszul, significant challenges emerge, and examples of minimal models in this setting remain scarce. 

Recently, 
cohomology theory of Nijenhuis algebras was studied in \cite{das1,Song},
it is very exciting to investigate cohomology and deformation of Nijenhuis operator on associative algebra morphisms.
This was our motivation for writing the present paper.
Precisely, the concepts of bimodule  over Nijenhuis algebra morphisms are defined and cohomology and deformation of Nijenhuis algebra morphisms may be studied.  In development, ($\mathrm{CCT}$) the cohomology comparison Theorem of Nijenhuis versions is considered, which says that  the cohomology of Nijenhuis algebra morphisms is isomorphic to the cohomology of an auxiliary Nijenhuis algebra. Finally, we give a minimal model of operad for Nijenhuis algebras
morphisms.\\ 
\textit{The paper is organized as follows:} In Section \ref{sec:2} reviews some concepts on Nijenhuis algebras and Nijenhuis bimodule. In Section \ref{sec:3} we introduce cohomology of
Nijenhuis algebra morphisms which
involves both Nijenhuis algebra part and morphism part (see Definition \ref{defi: cohomology of diff morphism}
). In Section \ref{sec:4} sets up deformation of Nijenhuis algebra morphisms, where it is shown that a Nijenhuis algebra moprhisms is rigid if the $2$nd-cohomology group is trivial (see Theorem \ref{thm: rigid}). Section \ref{sec:5} gives ($\mathrm{CCT}$) the cohomology comparison Theorem of Nijenhuis algebra morphisms, which says a cohomology of a Nijenhuis algebra morphisms is isomorphic to the cohomology of an auxiliary Nijenhuis algebra (see Theorem \ref{thm: CCT of RB}). In Section \ref{sec:6}
we present a minimal model of operad for Nijenhuis algebras morphisms via the method of Dotsenkol-Poncin \cite{Do} (see Proposition \ref{min-model}) . 

Throughout this paper, $\bk$ denotes a field. All the vector spaces and algebras are over $\bk$ and all tensor products are also taking over $\bk$.

\section{Nijenhuis algebra and their bimodule} \label{sec:2}
In this section, we discuss Nijenhuis algebras and their bimodule and we provide some basic observations.

\begin{defn}
	Let $(A, \cdot)$ be an associative algebra. A linear operator $P_A: A \rightarrow A$ is a \textit{Nijenhuis operator} if it satisfies
	\begin{eqnarray*} \label{Eq: Nijenhuis relation}
		P_A(x) \cdot P_A(y) = P_A \big(P_A(x) \cdot y + x \cdot P_A(y) - P_A( x \cdot y)\big),~\text{for any } x, y \in A.
	\end{eqnarray*}	
	 
In this case, $(A, P_A)$ is called a \textit{Nijenhuis algebra}.
\end{defn}
 
Next, we recall a morphism between two Nijenhuis algebras.
\begin{defn}
Let $(A, P_{A})$ and $(B, P_{B})$ be two Nijenhuis algebras. A morphism of associative algebras $f: A \rightarrow B$ is called a \textit{morphism of Nijenhuis algebras}, if $ f \circ P_{A}=P_{B} \circ f $.
 \end{defn}



To better illustrate the above definition, we now present an explicit example of morphism between two Nijenhuis algebras.

\begin{exam}
We consider a $3$-dimensional Nijenhuis  algebra $(A,P_A)$
with respect to a basis $\{e_1, e_2, e_3\}$, by the multiplication $\cdot_A$ and the Nijenhuis operator $P_A$ 
such that
\[
e_1 \cdot_A e_1 = \alpha e_1, \quad e_1 \cdot_A e_2 = e_2 \cdot_A e_1 = \alpha e_2, \quad e_1 \cdot_A e_3 = e_3 \cdot_A e_1 = \beta e_3,
\]
\[
e_2 \cdot_A e_2 = \alpha e_2, \quad e_2 \cdot_A e_3 = \beta e_3, \quad e_3 \cdot_A e_2 = e_3 \cdot_A e_3 = 0.
\]
\[
P_A(e_1) = \alpha e_1, \quad P_A(e_2) = \alpha e_2, \quad P_A(e_3) = \beta e_3.
\]
where $\alpha,\beta$ are parameters.\\
We consider also a $2$-dimensional Nijenhuis algebra $(B,P_B)$ defined, with respect to a
basis $\{f_1, f_2\}$, by the multiplication $\cdot_B$ and the Nijenhuis operator $P_B$ such that
\[
f_1 \cdot_B f_1 = f_1, \quad f_i \cdot_B f_j = f_2, \quad \text{for} \quad (i,j) \neq (1,1).
\]

\[
P_B(f_1) = \gamma f_1-\gamma f_2, \quad P_B(f_2) = 0.
\]
where $\gamma$ is a parameter.

Then, the linear map \(\phi : (A,P_A) \to (B,P_B)\) defined, when $\alpha=\beta=\gamma=1$ as:
\[
\phi(e_1) = f_1-f_2, \quad \phi(e_2) = f_1-f_2, \quad \phi(e_3) = 0,
\]
is a morphism of Nijenhuis algebras.    
\end{exam}
 \begin{defn}
 	Let $(A, P_A)$ be a Nijenhuis associative algebra.
 	\begin{itemize}
 		\item[(i)] A \textit{bimodule over the Nijenhuis algebra} $(A, P_A)$ is a bimodule $M$ over associative algebra $(A,\cdot)$ endowed with an operator $P_M: M\rightarrow M$, such that for any $x\in A, m\in M$, the following equalities hold:
 		\begin{eqnarray}
P_A(x) P_M(m) =& P_M \big(P_A(x)m + xP_M(m) - P_M (xm)\big),\\
P_M(m) P_A(x) =& P_M \big(P_M(m)x + mP_A(x) - P_M (mx)\big).
 		\end{eqnarray}
 		\item[(ii)] Given two bimodule $(M,P_M)$ and $(N,P_N)$ over Nijenhuis algebra $(A,P_A)$, a morphism from $(M,P_M)$ to $(N,P_N)$ is a bimodule morphism $f: M\rightarrow N$ over $(A, \cdot)$ such that:
 	\begin{equation}\label{bimodule}
        f\circ P_M=P_N\circ f.
        \end{equation}
 	\end{itemize}
 \end{defn}
\begin{exam}
Any Nijenhuis algebra $(A, P_A)$ is a Nijenhuis bimodule over itself, called the regular Nijenhuis bimodule.
\end{exam}
\medskip


It is important to remark that Nijenhuis operators on an associative algebra $(A, \cdot )$ are closely related
to linear deformation of the algebra structure \cite{Gerst3}.
\begin{prop}\label{newmultiplication}
(\cite{Song}) 
	\label{Prop: new Nj algebra}
	Let $(A, P_A)$ be a Nijenhuis algebra. Define a new binary operation (deformed multiplication) $\cdot_P$ over $A$ as:
    
	\begin{eqnarray*}
		x \cdot_P y := P_A(x)\cdot y + x\cdot P_A(y) - P_A(x\cdot y),
	\end{eqnarray*}
	for any $x, y\in A$. Then
	\begin{enumerate}
		\item the operation $\cdot_P $ is associative and $(A, \cdot_P)$ is a new associative algebra;
		\item the triple $(A, \cdot_P, P_A)$ also forms a Nijenhuis algebra and denote it by $(A_P, P_A)$.
	\end{enumerate}
\end{prop}	
Let $(A, P_A)$ be a Nijenhuis algebra and $(M,P_M)$ be a bimodule over $(A, P_A)$. Then
we can make $M$ into a bimodule over $(A_P, P_A) $. For any $x\in A$ and $m\in M$, we define the following actions:
\begin{align*}
x\rhd m:=&P_A(x)m+xP_M(m)-P_M(xm),\\
m\lhd x:=&P_M(m)x+m P_A(x)-P_M(mx).
\end{align*}
Then we have the following Lemma. 
\begin{lemma}\label{Prop: new-bimodule}
		The action  $(\rhd, \lhd)$ makes $M$ into a bimodule over the \textit{deformed associative algebra} $(A_P, P_A)$ and we denote this bimodule by $({}_{\rhd}M_\lhd, P_M)$. 
\end{lemma}
Gerstenhaber and Schack \cite{Gerst2} define the bimodule of an associative algebra morphisms, that is, let $\phi:A\rightarrow B$ be a morphism of associative algebras , then a $\phi$-bimodule is a triple $\langle M,N ,\psi\rangle$ such
that $M$ is a  bimodule over $A$, $N$ is a  bimodule over $B$, and $\psi: M \rightarrow N$
is an $A$-bimodule morphism, where $N$ is considered as a bimodule over associative algebra $A$ in a natural way. Similarly, let $\phi:(A, P_{A})\rightarrow (B, P_{B})$ be a morphism of Nijenhuis algebra, then a Nijenhuis $\phi$-bimodule is a triple $\langle (M, P_M), (N, P_N),\psi\rangle$ such
that $(M,P_M)$ is a Nijenhuis bimodule over $(A, P_{A})$, $(N,P_N)$ is a Nijenhuis bimodule over $(B, P_{B})$, and $\psi: (M,P_M)\rightarrow (N,P_N)$
is an $(A, P_{A})$-bimodule morphism, where $(N,P_N)$ is regarded as a bimodule over Nijenhuis algebra $(A, P_{A})$ in a natural way.
Then we have the
following observation, which may play an important role to study the cohomology
theory of Nijenhuis algebra morphisms in the next section.
\begin{thm}\label{thm:triang action Module-morphism}
Let $\psi:(M,P_M)\rightarrow(N,P_N)$ be a Nijenhuis morphism of $(A, P_A)$-bimodule, 
then $\psi:({}_{\rhd}M_\lhd,P_M) \rightarrow({}_\rhd N_\lhd,P_N) $ 
is a Nijenhuis morphism of $(A_P, P_A)$-bimodule. 
\end{thm}
\begin{proof}
Let $({}_{\rhd}M_\lhd, P_M)$, $({}_{\rhd}N_\lhd, P_N)$ be two $(A, P_A)$-bimodules. It suffice to show that
$\psi:({}_{\rhd}M_\lhd,P_M)$ $ \rightarrow({}_\rhd N_\lhd,P_N) $ is a morphism of $(A_P, P_A)$-bimodule .
Since $\psi:(M,P_M)\rightarrow(N,P_N)$ is a Nijenhuis morphism we have $\psi \circ P_M=P_N \circ \psi$, it follows that, for any $x \in A, m \in M$,  
\begin{align*}
\psi( x \rhd_M m)=&\psi(P_A(x)m+xP_M(m)-P_M(xm))\\
=&\psi(P_A(x)m)+\psi(xP_M(m) )-\psi(P_M(xm))\\
                 =&P_A(x )\psi(m) +x P_N(\psi(m) )-P_N(x\psi(m) )  \\
                =&x\rhd_N \psi(m). 
\end{align*}
Similarly, we can show that $\psi(m \lhd_M x)= \psi(m) \lhd_N x$.
Then $\psi:({}_{\rhd}M_\lhd,P_M) \rightarrow({}_\rhd N_\lhd, P_N) $ 
is a Nijenhuis morphism of $(A_P, P_A)$-bimodule.
\end{proof}
\begin{remark}\label{remark morphism module}
Let $\phi:(A, P_{A})\rightarrow (B, P_{B})$ be a morphism of Nijenhuis algebras and $\langle (M,P_M), (N,$\\$ P_N),\psi\rangle$ be a Nijenhuis $\phi$-bimodule. By Theorem \ref{thm:triang action Module-morphism},
 $\psi:({}_{\rhd}M_\lhd,P_M)\rightarrow({}_{\rhd}N_\lhd,P_N)$ is a Nijenhuis morphism
 of $(A_P, P_A)$-bimodule 
 that is, $\langle ({}_{\rhd}M_\lhd, P_M),({}_\rhd N_\lhd,P_N) ,\psi\rangle$ is still a Nijenhuis $\phi$-bimodule.
To avoid confusion in the following context, we denote it by
 $\langle{}_{\rhd}M_\lhd,{}_\rhd N_\lhd, {}_\rhd\psi_\lhd\rangle$.
\end{remark}
\section{Cohomology of Nijenhuis algebra morphisms } \label{sec:3}

Let $M$ be a bimodule over an associative algebra $A$. Recall that the \textit{Hochschild cohomology of $A$ with coefficients in $M$}:  $$(C_{\mathrm{Alg}}^\bullet(A,M):=\bigoplus\limits_{n=0}^\infty C_{\mathrm{Alg}}^n(A,M), \delta^n_{\mathrm{Alg}}),$$ where $C_{\mathrm{Alg}}^n(A,M)=(\mathrm{Hom}(A^{\otimes n},M)$ and the differential
$\delta^n_{\mathrm{Alg}}: C^n_{\mathrm{Alg}}(A,M)\rightarrow C^{n+1}_{\mathrm{Alg}}(A,M)$  is given by
\begin{align*}\delta^n_{\mathrm{Alg}}(f)(x_1,\dots, x_{n+1}):=&x_1f(x_2, \dots, x_n)+\sum_{i=1}^n(-1)^if(x_1, \dots, x_ix_{i+1}, \dots, x_{n+1})\\
& +(-1)^{n+1}f(x_1,\dots, x_{n})x_{n+1}
\end{align*}
for all $f\in C^n(A,M), x_1,\dots,x_{n+1}\in A$. The corresponding \textit{Hochschild cohomology} is denoted $\mathrm{HH}^\bullet(A,M)$. When $M=A$, just denote the \textit{Hochschild cochain complex with coefficients in $A$} by $C^\bullet_{\mathrm{Alg}}(A)$ and denote the \textit{Hochschild cohomology by $\mathrm{HH}^\bullet(A)$.}
\medskip

Now, we recall the cohomology theory of Nijenhuis algebras \cite{das1,Song}.
\begin{defn}
 	Let $A=(A, P_A)$ be a Nijenhuis algebra and $M=(M,P_M)$ be a Nijenhuis bimodule over it. Then the cochain complex $(C^\bullet(A, {_\rhd M_\lhd}),\delta_{\mathrm {Alg}} ^\bullet)$ is called the \textit{cochain complex of Nijenhuis operator $P_A$ with coefficients in $M$},  denoted it by $(C_{\mathrm{NjO}}^\bullet(A, M), \delta_{\mathrm{NjO}}^\bullet)$, 
	as follows:
	for any $n\geqslant 0 $,
	$$C_{\mathrm{NjO}}^n(A, M):=\Hom(A^{\otimes n}, M)$$
	and its differential	$\delta_{\mathrm{NjO}}^n: C^n_{\mathrm{NjO}}(A, M )\rightarrow C^{n+1}_{\mathrm{NjO}}(A, M) $
	is given by
	\begin{align*} \label{Eq: diff of Nijenhuis operator}
		\delta^n_{\mathrm{NjO}}(f) := - P_M\circ \delta_{\mathrm{Alg}}^{n}(f)+\partial^{n}(f)
	\end{align*}
	for any $f\in C_{\mathrm{NjO}}^n(A, M)$, 
  where $\partial^n: C^n_{\mathrm{Alg}}(A_P, {M_P})\rightarrow C^{n+1}_{\mathrm{Alg}}(A_P, {M_P})$
is given by
\begin{align*}
	\partial^n(f)(x_{1, n+1})
	=&\ x_1\rhd f(x_{2, n+1})+\sum_{i=1}^n (-1)^{i} f (x_{1, i-1}\otimes x_{i}\cdot_P x_{i+1} \otimes x_{i+2, n+1})
	+ (-1)^{n+1} f(x_{1, n})\lhd x_{n+1}. 
    \end{align*} 
\end{defn}
Let $A=(A, P_A)$ be a Nijenhuis algebra and $M=(M, P_M)$ be a Nijenhuis bimodule over it. 

 Given a chain map: $\Phi^\bullet: C_{\mathrm{Alg}}^\bullet(A,M)\rightarrow C^\bullet_{\mathrm{NjO}}(A,M) $, as following
\begin{align*}
\Phi^0:=\mathrm{id}_M,~~\Phi^1(f):= f\circ P_A-P_M\circ f
\end{align*} 

and when $n\geqslant 2$, 
\begin{equation}\label{cochai}
\begin{split}
\Phi^n&(f)(x_1, \dots, x_n)  
\\ 
&:= \sum_{k=0}^n \sum_{1\leqslant i_1<i_2<\dots<i_k\leqslant n} (-1)^{n-k}
P_M^{n-k} \circ f(x_{1, i_1-1}, P(x_{i_1}), x_{i_1+1, i_2-1}, P(x_{i_2}), \dots, P(x_{i_k}), x_{i_k+1, n}), 
\end{split}
\end{equation}
for any $ f\in C^n_{\mathrm{Alg}}(A, M)$.

\begin{defn} \label{Def: definition of Nijenhuis cohomology}
	Let $M=(M, P_M)$ be a Nijenhuis bimodule over a Nijenhuis algebra $A=(A, P_A)$. We define the \textit{ cochain complex of Nijenhuis algebra $A$ with coefficients in Nijenhuis bimodule $M$}, denoted by $(C^\bullet_{\mathrm{NjA}}(A, M), \delta_{\mathrm{NjA}}^\bullet)$, to be the mapping cone of $\Phi^\bullet$ shifted by $-1$, that is,
	$$C^0_{\mathrm{NjA}}(A, M):=C^0_\mathrm{Alg}(A, M) \quad \mathrm{and} \quad C^n_{\mathrm{NjA}}(A, M):=C^n_\Alg(A, M)\oplus C^{n-1}_{\mathrm{NjO}}(A, M), \text{ for } n\geqslant 1, $$
	and its differential $\delta_{\mathrm{NjA}}^n: C^n_{\mathrm{NjA}}(A, M)\rightarrow C^{n+1}_{\mathrm{NjA}}(A, M)$ is given by
	\begin{align*}
		\delta_{\mathrm{NjA}}^n(f, g)
		:=&\ (\delta_{\mathrm{Alg}}^n(f), -\delta_\mathrm{NjO}^{n-1}(g)-\Phi^n(f))
	\end{align*}	
	for any $f\in C^n_\mathrm{Alg}(A, M)$ and $g\in C^{n-1}_{\mathrm{NjO}}(A, M)$.

	The cohomology of $C^\bullet_{\mathrm{NjA}}(A, M)$, denoted by $\mathrm{H}_{\mathrm{NjA}}^\bullet(A, M)$, is called the \textit{cohomology of Nijenhuis algebra $A$ with coefficients in Nijenhuis bimodule $M$}.
	When $M$ is the regular Nijenhuis bimodule $A$ itself,
we simply denote $C^\bullet_{\mathrm{NjA}}(A, A)$ by $C^\bullet_{\mathrm{NjA}}(A)$ and $\mathrm{H}^\bullet_{\mathrm{NjA}}(A, A)$ by $\mathrm{H}^\bullet_{\mathrm{NjA}}(A)$, called the \textit{cochain complex of Nijenhuis algebra $A$} and the \textit{cohomology of Nijenhuis algebra $A$}, respectively.
\end{defn}

The original cohomology theory associated to deformation  of associative algebra morphism was introduced by Gerstenhaber \cite{Gerst1}.
Let $\phi: A \rightarrow B$ be a morphism of associative algebras and
$\langle M,N,\psi \rangle$ be a
$\phi$-module. The \textit{cochain complex $C_{\mathrm{mor} } ^\bullet (\phi,\psi)$ of $\phi$ with coefficients in
$\langle M,N,\psi \rangle$} is given by
\begin{equation*}\label{eq:dac}
	C_{\mathrm{mor}}^n(\phi,\psi):=
		\begin{cases}
            C^n (\phi,\psi)=0, &n\leq 0,\\
			C^n_{\mathrm{Alg}} (A ,M )\oplus C^{n}_{\mathrm{Alg}} (B ,N )\oplus C^{n-1}_{\mathrm{Alg}} (A ,N),&n\geq1, 
		\end{cases}
	\end{equation*}
and the coboundary operator $\delta_{\mathrm{mor}}^n:C_{\mathrm{mor}}^n(\phi,\psi) \rightarrow C_{\mathrm{mor}}^{n+1}(\phi,\psi)$ is given by
$$\delta_{\mathrm{mor}}^n(f,g,h):=(\delta_{\mathrm{Alg} }^n (f),\delta_{\mathrm{Alg}}^n (g),\psi \circ f-g\circ \phi^{\otimes n}-\delta_{\mathrm{Alg}}^{n-1}(h)).$$
Let $\phi:(A, P_{A})\rightarrow (B, P_{B})$ be a morphism of Nijenhuis algebras and $\langle (M,P_M), (N, P_N),\psi\rangle$ be a Nijenhuis $\phi$-bimodule. By Remark \ref{remark morphism module}, 
$\langle ({}_{\rhd}M_\lhd, P_M),({}_\rhd N_\lhd,P_N) ,{}_\rhd \psi_\lhd\rangle$ is still a Nijenhuis $\phi$-bimodule.
Hence, we may construct a cochain complex that controls
deformation of Nijenhuis algebra morphisms as follows.
\begin{prop}\label{pro:chain map}
Suppose that $C_{\mathrm{mor}}^\bullet(\phi,\psi)$ is a Hochschild cochain complex of $\phi$ with coefficients in $\psi$. Given a linear map 
$\Theta^\bullet$: $C_{\mathrm{mor}}^\bullet(\phi,\psi)\rightarrow
C_{\mathrm{mor}}^\bullet(\phi, {} _\rhd\psi_\lhd)$ by
\begin{enumerate}
\item[(i)] $\Theta^0: C_{\mathrm{mor}}^0(\phi,\psi)\rightarrow C_{\mathrm{mor}}^0(\phi,{} _\rhd\psi_\lhd)$ to
 be identity map,
\item[(ii)]
if $n\geqslant 1$, $\Theta^n: C_{\mathrm{mor}}^n(\phi,\psi)\rightarrow C_{\mathrm{mor}}^n(\phi,_{\rhd} \psi_\lhd)$ is defined as
$$ \Theta^n(f,g,h):=(\Phi_{A,M}^n(f),\Phi_{B,N}^n(g),\Phi_{A ,N}^{n-1}(h)),$$
 \end{enumerate}
 for $f\in C _{\mathrm{Alg}}^n(A ,M)$, $g\in C_{\mathrm{Alg}}^n(B ,N)$ and
 $h\in C_{\mathrm{Alg}}^{n-1}(A ,N)$. 
 Then $\Theta^\bullet: C_{\mathrm{mor}}^\bullet(\phi,\psi)\rightarrow C_{\mathrm{mor}}^\bullet(\phi, {}_\rhd \psi_\lhd)$ is a chain map, i.e., $\Theta^{n+1} \circ \delta_{\mathrm{mor}}^n=\delta_{\mathrm{mor}}^{n}\circ \Theta^n$.
In other words, the following diagram is commutative:
\[\xymatrix{
		&C_{\mathrm{mor}}^n(\phi,\psi)\ar[r]^-{\delta_{\mathrm{mor}}^n}\ar[d]^-{\Theta^n}&C_{\mathrm{mor}}^{n+1}(\phi,\psi)\ar[d]^{\Theta^{n+1}}&\\
		& C_{\mathrm{mor}}^n(\phi,{} _\rhd\psi_\lhd)\ar[r]^-{\delta_{\mathrm{mor}}^n}&C_{\mathrm{mor}}^{n+1}(\phi, {} _\rhd\psi_\lhd)&
}
\]
\end{prop}

\begin{proof}
Only need to show
$\Theta^{n+1} \circ \delta_{\mathrm{mor}}^n(f,g,h)=\delta_{\mathrm{mor}}^n \circ \Theta^n(f,g,h)$, for all $(f,g,h)\in C_{\mathrm {mor} } ^n(\phi,\psi)$, i.e.,
\begin{align*}
&\Theta^{n+1} \circ \delta_{\mathrm{mor}}^n(f,g,h)\\
=&~(\Phi_{A,M}^{n+1}(\delta_{\mathrm{Alg}}^n(f)),\Phi_{B,N}^{n+1}(\delta_{\mathrm{Alg}}^n(g)),
\Phi_{A,N}^{n}(\psi\circ f)-\Phi_{A,N}^{n}( g\circ \phi^{\otimes n})-\Phi_{A,N}^{n}(\delta_{\mathrm{Alg}}^{n-1}h))\\
=&~(\delta_\mathrm{NjO}^n(\Phi_{A,M}^{n}(f)),\delta_\mathrm{NjO} ^n(\Phi_{B,N}^{n}(g)),\psi\circ
(\Phi_{A,M}^{n}(f))-(\Phi_{B,N}^{n}(g))\circ \phi^{\otimes n}-\delta_\mathrm{NjO} ^{n-1}(\Phi_{A,N}^{n-1}(h)))\\
=&~\delta_{\mathrm{mor}}^n \circ \Theta^n(f,g,h).
\end{align*}
Since $\Phi^\bullet$ is a chain map, we only need to observe that,
from $P _N\circ \psi=\psi \circ P _M$, we have
$$\begin{aligned}
&~~~~\Phi_{A,N}^n(\psi \circ f)(x_1, \dots, x_n)\\
	&= \sum_{k=0}^n \sum_{1\leqslant i_1<i_2<\dots<i_k\leqslant n} (-1)^{n-k}
	P_N^{n-k} \circ (\psi \circ  f)(x_{1, i_1-1}, P_A(x_{i_1}), x_{i_1+1, i_2-1}, P_A(x_{i_2}), \dots, P_A(x_{i_k}), x_{i_k+1, n})\\
	&= \psi \Big(\sum_{k=0}^n \sum_{1\leqslant i_1<i_2<\dots<i_k\leqslant n} (-1)^{n-k}
	P_M^{n-k} \circ f(x_{1, i_1-1}, P_A(x_{i_1}), x_{i_1+1, i_2-1}, P_A(x_{i_2}), \dots, P_A(x_{i_k}), x_{i_k+1, n})\Big)\\
  &= \psi \circ (\Phi_{A,M}^n(f)) (x_1,\dots,x_{n}),
   \end{aligned}$$
and from $\phi \circ P _A =P _B \circ \phi$, we have
$$\begin{aligned}
&~~~~\Phi^n_{A,N}(g \circ \phi^{\otimes n})(x_1,\dots,x_{n})\\
&= \sum_{k=0}^n \sum_{1\leqslant i_1<i_2<\dots<i_k\leqslant n} (-1)^{n-k}
	P_N^{n-k} \circ (g \circ \phi^{\otimes n})(x_{1, i_1-1}, P_A(x_{i_1}), x_{i_1+1, i_2-1}, P_A(x_{i_2}), \dots, P_A(x_{i_k}), x_{i_k+1, n})\\
    &= \sum_{k=0}^n \sum_{1\leqslant i_1<i_2<\dots<i_k\leqslant n} (-1)^{n-k}
	P_N^{n-k} \circ g (\phi^{\otimes (i_1-1)}( x_{1, i_1-1}), \phi(P_A(x_{i_1})), \phi^{\otimes (i_2-i_1-1)}(x_{i_1+1, i_2-1}), \phi(P_A(x_{i_2})),  \\
&~~~~\dots, \phi(P_A(x_{i_k})),\phi^{\otimes (n-i_k)}( x_{i_k+1, n}))\\
    &= \sum_{k=0}^n \sum_{1\leqslant i_1<i_2<\dots<i_k\leqslant n} (-1)^{n-k}
	P_N^{n-k} \circ g (\phi^{\otimes (i_1-1)}( x_{1, i_1-1}), P_B(\phi(x_{i_1}), \phi^{\otimes (i_2-i_1-1)}( x_{i_1+1, i_2-1}), P_B(\phi(x_{i_2})),   \\
&~~~~\dots,P_B(\phi(x_{i_k})),\phi^{\otimes (n-i_k)}( x_{i_k+1, n}))\\
&=(\Phi^n_{B,N}(g)) \circ \phi^{\otimes n}(x_1,\dots,x_{n}).
\end{aligned}$$
\end{proof}
To conclude this section, we define cohomology of Nijenhuis algebra morphisms with coefficients in their bimodule.
\begin{defn}\label{defi: cohomology of diff morphism}
Let $\Theta^\bullet: C_{\mathrm{mor}}^\bullet(\phi,\psi)\rightarrow C_{\mathrm{mor}}^\bullet(\phi, {} _\rhd \psi_\lhd)$ be a chain map defined in Proposition
{\ref{pro:chain map}}. We may define a cochain complex $C_{\mathrm{NjM}}^\bullet(\phi,\psi)$ to be the negative shift of the
 mapping cone of $\Theta^{\bullet}$, that is,
\begin{enumerate}
\item[(i)] $C_{\mathrm{NjM}}^0(\phi,\psi):=C_{\mathrm{mor}}^0(\phi,\psi)$,
\item[(ii)]when $n\geqslant 1$, $C_{\mathrm{NjM}}^n(\phi,\psi):=C_{\mathrm{mor}}^n(\phi,\psi)\oplus
C_{\mathrm{mor}}^{n-1}(\phi, {} _\rhd \psi_\lhd)$,
\end{enumerate}
and the coboundary operator
$\mathrm{D}^n :C_{\mathrm{NjM}}^n(\phi,\psi)\rightarrow C_{\mathrm{NjM}}^{n+1}(\phi,\psi)$ is defined by\\
when $n=1$,
\begin{align*}
\mathrm{D}^1\Big((f ,g ,h),(m_1,m_2)\Big)
:=&\Big(\delta_{\mathrm{mor}}^1(f ,g ,h ), \delta_{\mathrm{mor}}^0(m_1, m_2 )-\Theta^1(f ,g ,h)\Big)\\\nonumber
:=&\Big((\delta_{\mathrm{Alg}}^1(f ),\delta_{\mathrm{Alg}}^1(g ),\psi\circ f -g \circ\phi),\\
&~(-\Phi_{A,M}^1 (f ),-\Phi_{B,N}^1 (g ), \psi(m_1)-m_2-\Phi^0_{A,N}(h) \Big),
\end{align*}
when $n\geqslant 2$,
\begin{align*}
\mathrm{D}^n\Big((f _1,g _1,h _1),(f _2,g _2,h _2)\Big)
:=&\Big(\delta_{\mathrm{mor}}^n(f _1,g _1,h _1),\delta_{\mathrm{mor}}^{n-1}(f _2,g _2,h _2)+(-1)^n\Theta^n(f _1,g _1,h _1)\Big)\\ \nonumber
:=&\Big((\delta_{\mathrm{Alg}}^n(f _1),\delta_{\mathrm{Alg}}^n(g _1),\psi\circ f _1-g _1\circ \phi ^{\otimes n}
-\delta_{\mathrm{Alg}}^{n-1}(h _1)),\\ \nonumber
&(\delta_{\mathrm{NjO}}^{n-1}(f _2)+(-1)^n\Phi_{A,M}^n (f _1) , \delta_{\mathrm{NjO}}^{n-1}(g _2)
+(-1)^n\Phi_{B,N}^n (g _1) , \\ \nonumber
&\psi\circ f _2-g _2\circ \phi^{\otimes n-1}-\delta_{\mathrm{NjO}}^{n-2}(h _2)+(-1)^{n}
\Phi_{A,N}^{n-1}(h _1)) \Big),
\end{align*}
for $(f _1,g _1,h _1)\in C_{\mathrm{mor}}^n(\phi,\psi), (f _2,g _2,h _2)\in C_{\mathrm{mor}}^{n-1}(\phi, {}_\rhd\psi_\lhd)$. 
where
$\Phi^\bullet$ is defined by Eq. \eqref{cochai}.
\end{defn}
\begin{prop}
With the above notations, we have $\mathrm{D}^{n+1} \circ \mathrm{D}^n=0$.
\end{prop}
\begin{proof}
For $(f _1,g _1,h _1)\in \mathrm{C}_{\mathrm{mor}}^n(\phi,\psi), (f _2,g _2,h _2)\in \mathrm{C}_{\mathrm{mor}}^{n-1}(\phi, {} _\rhd\psi_\lhd)$, we have
$$\begin{aligned}
&~~~~\mathrm{D}^{n+1} \circ \mathrm{D}^n\Big((f _1,g _1,h _1),(f _2,g _2,h _2)\Big)\\
&=\mathrm{D}^{n+1}\Big((\delta_{\mathrm{Alg} }^n(f _1),\delta_{\mathrm{Alg} }^n(g _1),\psi\circ f _1-g _1\circ \phi^{\otimes n}
-\delta_{\mathrm{Alg} }^{n-1}(h _1)\big),\\
&~~~~(\delta_\mathrm{NjO}^{n-1}(f _2)+(-1)^n\Phi_{A,M }^n (f _1) , \delta_\mathrm{NjO}^{n-1}(g _2)
+(-1)^n\Phi_{B,N }^n (g _1) , \\
&~~~~\psi\circ f _2-g _2\circ\phi^{\otimes n-1}-\delta_\mathrm{NjO}^{n-2}(h _2)+(-1)^{n}
\Phi_{A,N }^{n-1}(h _1)) \Big),\\
&=\Big((\delta_{\mathrm{Alg}}^{n+1}(\delta_{\mathrm{Alg}}^{n}(f _1)),\delta_{\mathrm{Alg}}^{n+1}(\delta_{\mathrm{Alg}}^{n}(g _1)),
\psi \circ \delta_{\mathrm{Alg}}^{n}(f _1)-\delta_{\mathrm{Alg} }^{n}(g _1)\circ \phi^{\otimes n+1}\\
&-\delta_{\mathrm{Alg}}^{n}(\psi \circ f _1)
+\delta_{\mathrm{Alg}}^{n}(g _1\circ \phi^{\otimes n})
+\delta_{\mathrm{Alg}}^{n}(\delta_{\mathrm{Alg}}^{n-1}(h _1))),\\
&~~~~(\delta_\mathrm{NjO}^n(\delta_\mathrm{NjO}^{n-1}(f _2))+(-1)^n\delta_{\mathrm{NjO}}^n(\Phi^n_{A,M }(f _1))
+(-1)^{n+1}\Phi_{A,M }^{n+1}(\delta_{\mathrm{Alg}}^n(f_1)),\\
&~~~~\delta_{\mathrm{NjO}}^n(\delta_\mathrm{NjO}^{n-1}(g_2))
+(-1)^n\delta_\mathrm{NjO}^n(\Phi^n_{B,N }(g _1))
+(-1)^{n+1}\Phi^{n+1}_{B,N}(\delta_{\mathrm{Alg}}^n(g _1)),\\
&~~~~\psi\circ( \delta_\mathrm{NjO}^{n-1}(f _2))+(-1)^n\psi \circ
(\Phi^n_{A,M }(f _1))-\delta_\mathrm{NjO}^{n-1}(g _2)\circ
\phi^{\otimes n}-(-1)^n\Phi^n_{B,N }(g _1)\circ \phi^{\otimes n}\\
&-\delta_\mathrm{NjO}^{n-1}(\psi\circ f _2)+\delta_\mathrm{NjO}^{n-1}(g _2\circ \phi^{\otimes n-1})
+\delta_\mathrm{NjO}^{n-1}(\delta_\mathrm{NjO}^{n-2}(h _2))-(-1)^{n}\delta_\mathrm{NjO}^{n-1}(\Phi_{A,N }^{n-1}
(h _1))\\
&+(-1)^{n+1}\Phi^n_{A,N }(\psi \circ f _1)-(-1)^{n+1}\Phi_{A,N}^n(g _1 \circ \phi^{\otimes n})-(-1)^{n+1} \Phi_{A,N }^n
(\delta_{\mathrm{Alg}}^{n-1}(h _1)))\Big).
\end{aligned}$$
To finish the proof, 
according to \cite[Theorem 2.4]{Mak}, we have 
$\psi \circ (\delta_{\mathrm{Alg}}^{n} (f_1))=\delta_{\mathrm{Alg}}^{n}  (\psi \circ f_1)$ and 
$(\delta_{\mathrm{Alg}}^{n} (g_1)) \circ \phi^{\otimes n+1}=
\delta_{\mathrm{Alg}}^{n}(
g_1\circ \phi^{\otimes n})$. Thus,
we only need to observe that, 
from Theorem \ref{thm:triang action Module-morphism}
it follows that
$\psi \circ (\delta_{\mathrm{NjO}}^{n-1} (f_2))=\delta_{\mathrm{NjO}}^{n-1}  (\psi \circ f_2).$ 
Additionally, from
$g_2 \circ \phi (x_1,\dots,x_{n})=g_2 \circ (\phi(x_2),\dots,\phi(x_{n}))$, we deduce that $(\delta_{\mathrm{NjO}}^{n-1} (g_2)) \circ \phi^{\otimes n}=
\delta_{\mathrm{NjO}}^{n-1}(
g_2\circ \phi^{\otimes n-1})$. 
\end{proof}
It follows from the above proposition that $\{C_{\mathrm{NjM}} ^\bullet(\phi,\psi),\mathrm{D}^\bullet\}$
is a \textit{cochain complex}. Let $\mathrm{Z}_{\mathrm{NjM}}^n(\phi,\psi)$ denote the \textit{space of $n$-cocycles} and
$\mathrm{B}_{\mathrm{NjM}}^n(\phi,\psi)$ denote the \textit{space of $n$-coboundaries}. Then we have
$\mathrm{B}_{\mathrm{NjM}}^n(\phi,\psi) \subset \mathrm{Z}_{\mathrm{NjM}}^n(\phi,\psi)$.
The corresponding quotient
$$\mathrm{H}_{\mathrm{NjM}}^n(\phi,\psi):=\frac{\mathrm{Z}_{\mathrm{NjM}}^n(\phi,\psi)}{\mathrm{B}_{\mathrm{NjM}}^n(\phi,\psi)},~~\text{for}~~n\geq0$$

are called the \textit{cohomology groups of the Nijenhuis algebra morphisms
$\phi:(A , P_A ) \to (B, P_B)$ with
coefficients in the bimodule $\langle (M,P _M), (N, P _N), {}_\rhd\psi_\lhd\rangle $}.\\

In the following, we provide a sufficient condition for the vanishing of the cohomology groups $\mathrm{H}^\bullet_\mathrm{NjM}(\phi, \psi)$
in terms of the vanishing of some cohomology groups of Nijenhuis algebras.
\begin{prop}\label{vanish-prop}
If $\mathrm{H}^n_\mathrm{NjA} (A, M)$, $\mathrm{H}^n_\mathrm{NjA} (B, N)$ and $\mathrm{H}^{n-1}_\mathrm{NjA} (A, N)$ are trivial, so is $\mathrm{H}^n_\mathsf{NjM} (\phi, \psi)$.
\end{prop}
\begin{proof}
Let $\alpha=\big((f_1, g_1, h_1),(f_2, g_2, h_2)\big) \in \mathrm{Z}^n_\mathsf{NjM}(\phi, \psi)$ be an $n$-cocycle. 
Then by Definition \ref{defi: cohomology of diff morphism}
 and the hypothesis, one has that 
 $(f_1,f_2)=\delta^{n-1}_\mathrm{NjA} (f'_1,f'_2)$ and $(g_1,g_2)=\delta^{n-1}_\mathrm{NjA} (g'_1,g'_2)$
for some
$(n-1)$-cochains $(f'_1,f'_2)\in \mathrm{C} _{\mathrm{NjA}}^{n-1}(A ,M)$ and $(g'_1,g'_2)\in \mathrm{C}_{\mathrm{NjA}}^{n-1}(B ,N)$. 
Since  
$\mathrm{D}^n= 0$, via \cite[Proposition 3.3]{Yau} that:
$$
0=\delta_\mathrm{NjO}^{n-1}\Big(
\psi \circ f'_1- g'_1\circ \phi^{\otimes n-1}-h _1 \Big).$$
Furthermore, 
it suffices to prove that:
\begin{align*}
0=&\psi\circ f _2-g _2\circ \phi^{\otimes n-1}+\delta_\mathrm{NjO}^{n-2}(h _2)-
\Phi_{A ,N}^{n-1}(h _1) \\
=&\psi\circ \Big(-\delta_\mathrm{NjO}^{n-2}(f' _2)-\Phi_{A ,M}^{n-1}(f'_1)\Big)-\Big(-\delta_\mathrm{NjO}^{n-2}(g '_2)-\Phi_{B ,N}^{n-1}(g'_1)\Big)\circ \phi^{\otimes n-1}\\
&+\delta_\mathrm{NjO}^{n-2}(h _2)-
\Phi_{A ,N}^{n-1}(h _1) \\
=&-\psi\circ \delta_\mathrm{NjO}^{n-2}(f'_2)-\psi\circ \Phi_{A ,M}^{n-1}(f'_1)+\delta_\mathrm{NjO}^{n-2}(g'_2)\circ \phi^{\otimes n-1}+\Phi_{B ,N}^{n-1}(g'_1)\circ \phi^{\otimes n-1}\\
&+\delta_\mathrm{NjO}^{n-2}(h _2)-
\Phi_{A ,N}^{n-1}(h _1) \\
=&-\delta_\mathrm{NjO}^{n-2}(\psi \circ f'_2)-\Phi_{A ,N}^{n-1}(\psi \circ f'_1)+\delta_\mathrm{NjO}^{n-2}(g'_2\circ \phi^{\otimes n-2})+\Phi_{A ,N}^{n-1}(g'_1\circ \phi^{\otimes n-1})\\
&+\delta_\mathrm{NjO}^{n-2}(h _2)-
\Phi_{A ,N}^{n-1}(h _1)\\
=&-\delta_\mathrm{NjO}^{n-2}\Big(\psi \circ f '_2-g'_2\circ \phi^{\otimes n-2}-h _2\Big)-
\Phi_{A ,N}^{n-1}\Big(\psi \circ f'_1- g'_1\circ \phi^{\otimes n-1}-h _1 \Big).
\end{align*}
In other words, 
$(\psi \circ f'_1- g'_1\circ \phi^{\otimes n-1}-h _1,\psi \circ f '_2-g'_2\circ \phi^{\otimes n-2}-h _2)$
is an $(n-1)$-cocycle.
It follows from the hypothesis
that $(\psi \circ f'_1- g'_1\circ \phi^{\otimes n-1}-h _1,\psi \circ f '_2-g'_2\circ \phi^{\otimes n-2}-h _2)=\delta^{n-2}_\mathrm{NjA} (h'_1,h'_2)$
for some $(n-2)$-cochain $(h'_1,h' _2)\in C^{n-2}_\mathrm{NjA} (A, N)$
and, hence $\alpha=\mathrm{D}^{n-1}\big((f'_1, g'_1, h'_1),(f'_2, g'_2, h'_2)\big)$.
\end{proof}

\smallskip

\section{Deformations of Nijenhuis algebra morphisms}\label{sec:4}
In this section, we study formal deformations of a Nijenhuis algebra morphisms and consider the
rigidity of Nijenhuis algebra morphisms.

Let $(A ,\mu_A ,P_A )$ and $(B ,\mu_B ,P_B )$
be Nijenhuis algebras and
$\phi:(A,P_A)\rightarrow (B,P_B)$ be a morphism of Nijenhuis algebras.
For $\mathfrak{X}=\{A ,B \}$, define
\begin{eqnarray*}
\mu_{\mathfrak{X},t}&:=&\sum_{i=0}^{\infty}\mu_{\mathfrak{X},i} t^i,~~~\mu_{\mathfrak{X},0}=\mu_{\mathfrak{X}},\\
P_{\mathfrak{X},t}&:=&\sum_{i=0}^{\infty}P_{\mathfrak{X},i} t^i,~~~P_{\mathfrak{X},0}=P_{\mathfrak{X}},\\
\phi_{t}&:=&\sum_{i=0}^{\infty}\phi_{i} t^i,~~~\phi_{0}=\phi, 
\end{eqnarray*}
then, $\phi_t:(\mu_{A ,t},P_{A ,t})\to (\mu_{B ,t},P_{B ,t})$ is called a
\textit{one-parameter formal deformation} of
 $\phi:(A,P_A) \to (B,P_B)$, if
 $(A [\![t]\!],\mu_{A ,t},P_{A ,t})$,
 $(B [\![t]\!],\mu_{B ,t},P_{B ,t})$
 are $\mathrm{k} [\![t]\!]$-Nijenhuis algebras and
 $\phi_{t}: A [\![t]\!]\rightarrow B [\![t]\!]$
is a morphism of Nijenhuis algebras.\\
A power series $\phi_t:(\mu_{A ,t},P_{A ,t})\to (\mu_{B ,t},P_{B ,t})$ is a
one-parameter formal deformation of
$\phi:(A,P_A) \to (B,P_B)$ if and only if the
following equations hold :
\begin{eqnarray*}
 \mu_{A ,t}(x_1 \otimes \mu_{A ,t}(x_2\otimes x_3))
 &=&\mu_{A ,t}(\mu_{A ,t}(x_1\otimes x_2)\otimes x_3),\\
 \mu_{A ,t}(P_{A ,t}(x_1)\otimes P_{A ,t}(x_2))
 &=&P_{A ,t}\Big(\mu_{A ,t}(P_{A ,t}(x_1)\otimes x_2)+\mu_{A ,t}(x_1 \otimes P_{A ,t}(x_2))-P_{A ,t}(x_1\otimes x_2)\Big),\\
 \mu_{B ,t}(y_1 \otimes\mu_{B ,t}(y_2 \otimes y_3))
 &=&\mu_{B ,t}(\mu_{B ,t}(y_1 \otimes y_2)\otimes y_3),\\
 \mu_{B ,t}(P_{B ,t}(y_1)\otimes P_{B ,t}(y_2))
 &=&P_{B ,t}\Big(\mu_{B ,t}(P_{B ,t}(y_1)\otimes y_2)+\mu_{B ,t}(y_1 \otimes P_{B ,t}(y_2))-P_{B ,t}(y_1\otimes y_2)\Big),\\
 \phi_{t} \circ \mu_{A ,t}(x_1 \otimes  x_2)
 &=&\mu_{B ,t}\circ (\phi_{t}(x_1)\otimes \phi_{t}(x_2)),\\
\phi_{t}\circ P_{A ,t}(x_1)&=& P_{B ,t} \circ \phi_{t}(x_1),
 \end{eqnarray*}
for all $x_1,x_2,x_3\in A $, $y_1,y_2,y_3\in B $.
By expanding these equations and comparing the coefficients of $t^n$, the following equations hold:
For any
$n\geqslant 0$,
\begin{align}
\sum_{i=0}^n\mu_{\mathfrak{X},i} \circ (\mathrm{id} \otimes \mu_{\mathfrak{X},n-i})=&\sum_{i=0}^n \mu_{\mathfrak{X},i}  \circ (\mu_{\mathfrak{X},n-i} \otimes \mathrm{id}),
\label{Eq: deform eq for  products in DA}\\
\sum_{\substack{i+j+k=n \\ i,j,k \geq 0}} \mu_{\mathfrak{X},i}\circ(P_{\mathfrak{X},j}\otimes P_{\mathfrak{X},k})=&\sum_{\substack{i+j+k=n \\ i,j,k \geq 0}}P_{\mathfrak{X},i}\circ (\mu_{\mathfrak{X},j}(P_{\mathfrak{X},k}\otimes \mathrm{id}))+\sum_{\substack{i+j+k=n \\ i,j,k \geq 0}}P_{\mathfrak{X},i}\circ (\mu_{\mathfrak{X},j}(\mathrm{id} \otimes P_{\mathfrak{X},k}))\\
&-\sum_{\substack{i+j+k=n \\ i,j,k \geq 0}}P_{\mathfrak{X},i}\circ P_{\mathfrak{X},j}\circ(\mu_{\mathfrak{X},k}(\mathrm{id} \otimes \mathrm{id})) \nonumber,\\
\sum_{i+j=n \atop i, j\geqslant 0}\phi_i \circ \mu_{A ,j}
=&\sum_{i+j+k=n \atop i, j, k \geqslant 0}\mu_{B ,i} \circ 
 (\phi_j \otimes \phi_k),\label{Eq: deform eq for  morphism of multipication}\\
\sum_{i=0}^n\phi_{i}\circ P_{A ,n-i}
=&\sum_{i=0}^nP_{B ,i}\circ \phi_{n-i},
\label{Eq: deform eq for  morphism of operators}
\end{align}
where $\mathfrak{X}=\{A ,B \}$.
\begin{prop}
Let $\phi_{t}:(\mu_{A ,t},P_{A ,t})\to (\mu_{B ,t},P_{B ,t})$ be a
one-parameter formal deformation of
$\phi: (A,P_A)\to (B,P_B)$, then
$\big((\mu_{A ,1},\mu_{B ,1},\phi_{1}),(P_{A ,1},P_{B ,1},0)\big)$
is a 2-cocycle in the cochain complex $C^\bullet_{\mathrm{NjM}}(\phi,\phi)$.
\end{prop}
\begin{proof}
We only need to show that
\begin{align*}
&\mathrm{D}^2\big((\mu_{A ,1},\mu_{B ,1},\phi_{1}),(P_{A ,1},P_{B ,1},0)\big)\\
=&\big((\delta_{\mathrm{Alg}}^2(\mu_{A ,1}),\delta_{\mathrm{Alg}}^2(\mu_{B ,1}),\phi \circ \mu_{A ,1}-\mu_{B ,1}\circ \phi^{\otimes 2}-\delta_{\mathrm{Alg}}^{1}(\phi_{1})),\\\nonumber
&(-\delta_{\mathrm{NjO}}^{1}(P_{A ,1})-\Phi_{A }^2 (\mu_{A ,1}) ,-\delta_{\mathrm{NjO}}^{1}(P_{B ,1})-\Phi_B ^2 (\mu_{B ,1}) ,-
\phi\circ P_{A ,1}+P_{B ,1}\circ \phi-
\Phi_{A ,B }^{1}(\phi_{1}) \big)\\
=&0.
\end{align*}
Due to deformations of algebra morphisms~\cite{Gerst1} and deformations of Nijenhuis algebras~\cite{das1}, it suffices to show $-\phi\circ P_{A ,1}+P_{B ,1}\circ \phi-\Phi_{A ,B }^{1}(\phi_{1}) =0.$

When $n=1$, Eq.~(\ref{Eq: deform eq for  morphism of operators}) becomes
$$
\phi \circ P_{A ,1}+
\phi_{1}\circ P_{A }
=P_{B }\circ \phi_{1}+P_{B ,1}\circ \phi,$$
that is, $-\phi\circ  P_{A ,1}+P_{B ,1} \circ \phi-\Phi_{A ,B }^{1}(\phi_{1}) =0$,
so  $\big((\mu_{A ,1},\mu_{B ,1},\phi_{1}),(P_{A ,1},P_{B ,1},0)\big)$ is a $2$-cocycle.
\end{proof}
\begin{defn}
The $2$-cocycle $\big((\mu_{A ,1},\mu_{B ,1},\phi_{1}),(P_{A ,1},P_{B ,1},0)\big)$ is called an \textit{infinitesimal deformation}
 of the one-parameter formal deformation
 $\phi_{t}:(\mu_{A ,t},P_{A ,t})\to (\mu_{B ,t},P_{B ,t})$.
\end{defn}
Let $\phi_{t}:(\mu_{A ,t},P_{A ,t})\to (\mu_{B ,t},P_{B ,t})$ and $\phi'_{t}:(\mu'_{A ,t},P'_{A ,t})\to (\mu'_{B ,t},P'_{B ,t})$ be two
one-parameter formal deformations of $\phi:(A , P_{A }) \to (B , P_{B })$. Then, a \textit{formal isomorphism} from
$\phi_{t}:(\mu_{A ,t},P_{A ,t})\to (\mu_{B ,t},P_{B ,t})$ to $\phi'_{t}:(\mu'_{A ,t},P'_{A ,t})\to (\mu'_{B ,t},P'_{B ,t})$ are two power series
$(F _{A ,t},F _{B ,t})$ with
\begin{align*}
F _{A ,t}~~&:=~~\sum_{i=0}^{\infty}F _{A ,i}t^i:A [\![t]\!]\rightarrow A [\![t]\!],~~F _{A ,i}
\in \mathrm{Hom}(A ,A ), F _{A ,0}=\mathrm{id}_A , \\
F _{B ,t}~~&:=~~\sum_{i=0}^{\infty}F _{B ,i}t^i:B [\![t]\!]\rightarrow B [\![t]\!],~~F _{B ,i}
\in \mathrm{Hom}(B ,B ), F _{B ,0}=\mathrm{id}_B ,
\end{align*}
such that for $\mathfrak{X}=\{A ,B \}$
\begin{align*}
F _{\mathfrak{X},t}\circ \mu_{\mathfrak{X},t}~~&=~~\mu'_{\mathfrak{X},t}\circ (F _{\mathfrak{X},t} \otimes F _{\mathfrak{X},t}),\\
F _{\mathfrak{X},t}\circ  P_{\mathfrak{X},t}~~&=~~P'_{\mathfrak{X},t}\circ F _{\mathfrak{X},t},\\
F _{B ,t}\circ \phi_t ~~&=~~\phi'_t \circ  F _{A ,t}.
\end{align*}
Now, we justify the cohomology theory of Nijenhuis algebra by interpreting 
lower-degree cohomology groups as formal deformations.
\begin{prop}
The infinitesimal deformations of two \textit{equivalent} one-parameter formal deformations of $\phi: (A,P_A)\to (B,P_B)$ are in the same cohomology class of $\mathrm{H}_{\mathrm{NjM}}^\bullet(\phi, \phi) $.
\end{prop}
\begin{proof}
Let $(F _{A ,t},F _{B ,t})$ be a formal isomorphism from
$\phi_{t}:(\mu_{A ,t},P_{A ,t})\to (\mu_{B ,t},P_{B ,t})$ to $\phi'_{t}:(\mu'_{A ,t},P'_{A ,t})\to (\mu'_{B ,t},P'_{B ,t})$, we have
\begin{align*}
\mu_{\mathfrak{X},1}-\mu'_{\mathfrak{X},1}~~&=~~\mu_\mathfrak{X}\circ (\mathrm{id} \otimes F _{\mathfrak{X},1})-F _{\mathfrak{X},1}\circ \mu_{\mathfrak{X}}+\mu_{\mathfrak{X}}\circ (F _{\mathfrak{X},1} \otimes \mathrm{id}),\\
P_{\mathfrak{X},1}-P'_{\mathfrak{X},1}~~&=~~P_{\mathfrak{X}}\circ F _{\mathfrak{X},1}-F _{\mathfrak{X},1}\circ P_{\mathfrak{X}},\\
\phi_{1}-\phi'_{1}~~&=~~\phi \circ F _{A ,1}-F _{B ,1}\circ  \phi.
\end{align*}
Thus,
$\big((\mu_{A ,1},\mu_{B ,1},\phi_{1}),(P_{A ,1},P_{B ,1},0)\big)-\big((\mu'_{A ,1},\mu'_{B ,1},\phi'_{1}),(P'_{A ,1},
P'_{B ,1},0)\big) =\mathrm{D}^1(F _{A ,1},F _{B ,1}).$
\end{proof}

\begin{defn}
A Nijenhuis algebra morphisms $\phi:(A ,P_A )\to (B ,P_B )$ is called \textit{rigid} if every one-parameter
formal deformation is equivalent to $\phi:(A ,P_A )\to (B ,P_B )$.
\end{defn}

The next result shows that if the $2$nd-cohomology group is trivial, then, the Nijenhuis algebra moprhisms is rigid .
\begin{theorem}\label{thm: rigid}
Let $\phi:(A ,P_A )\to (B ,P_B )$  be a Nijenhuis algebra morphisms. If $\mathrm{H}_{\mathrm{NjM}}^2(\phi,\phi)=0$, then $\phi:(A ,P_A )\to (B ,P_B )$  is rigid.
\end{theorem}
\begin{proof}
Let $\phi_{t}:(\mu_{A ,t},P_{A ,t})\to (\mu_{B ,t},P_{B ,t})$ be a one-parameter formal deformation of
$\phi: (A,P_A)\to(B,P_B)$, then 
$\big((\mu_{A ,1},\mu_{B ,1},\phi_{1}),
(P_{A ,1},P_{B ,1},0)\big)$ is a $2$-cocycle. By $\mathrm{H}_{\mathrm{NjM}}^2(\phi,\phi)=0$, there exists $\big((F '_1,G '_1,b),(a_1,b_1, 0)\big) \in
C_{\mathrm{NjM}} ^1(\phi,\phi)$ such that $\mathrm{D} \big((F '_1,G '_1,b),(a_1,b_1,0)\big)=\big((\mu_{A ,1},\mu_{B ,1},\phi_{1}),
(P_{A ,1},$\\$P_{B ,1},0)\big)$, that is,
\begin{eqnarray*}
\mu_{A,1}(x \otimes y)&=&xF'_1(y)-F'_1(xy)+F'_1(x)y, ~~   \forall x,y\in A\\
\mu_{B,1}(x \otimes y)&=&xG'_1(y)-G'_1(xy)+G'_1(x)y,~~   \forall x,y\in B\\
\phi_1(x)&=&\phi \circ F'_1(x)-G'_1 \circ \phi(x), ~~   \forall x \in A\\
P_{A,1}(x)&=&-\Phi^1 \circ F'_1(x)   ~~   \forall x\in A\\
P_{B,1}(y)&=&-\Phi^1 \circ G'_1(y),  ~~   \forall y\in B\\
P_B(b)&=&b_1-\phi(a_1).
\end{eqnarray*}
Set $F _{A ,t}=\mathrm{id}_A +F _{A ,1}t$, $F _{B ,t}=\mathrm{id}_B +F _{B ,1}t$, and define
$\phi'_{t}:(\mu'_{A ,t},P'_{A ,t})\to (\mu'_{B ,t},P'_{B ,t})$ by
\begin{align*}
\begin{split}
\mu'_{A ,t}~~&=~~F _{t}\circ\mu_{A ,t}\circ (F ^{-1}_{t} \otimes F ^{-1}_{t}),\\
P'_{A ,t}~~&=~~F _{t}\circ P_{A ,t}\circ F ^{-1}_{t},\\
\mu'_{B ,t}~~&=~~G _{t}\circ \mu_{B ,t}\circ (G ^{-1}_{t} \otimes G ^{-1}_{t}),\\
P'_{B ,t}~~&=~~G _{t}\circ P_{B ,t}\circ G ^{-1}_{t},\\
\phi'_t~~&=~~G _{t}\circ \phi_t \circ F ^{-1}_{t}.
\end{split}
\end{align*}
For $\mathfrak{X}=\{A ,B \}$, we have
\begin{align*}
\mu'_{\mathfrak{X},t}~~&=~~\mu_\mathfrak{X}+[F _{\mathfrak{X},1}\circ \mu_\mathfrak{X}+\mu_{A ,1}-\mu_A \circ (\mathrm{id}\otimes F _{\mathfrak{X},1}+F _{\mathfrak{X},1}\otimes \mathrm{id})]t+\mu'_{\mathfrak{X},2}t^2+\dots\\
&=~~\mu_\mathfrak{X}+\mu'_{\mathfrak{X},2}t^2+\dots, \\
P'_{\mathfrak{X},t}~~&=~~P_{\mathfrak{X}}+(F _{\mathfrak{X},1}\circ P_\mathfrak{X}+T_{\mathfrak{X},1}-P_\mathfrak{X}\circ F _{\mathfrak{X},1})t+P'_{\mathfrak{X},2}t^2+\dots\\
&=P_{\mathfrak{X}}+P'_{\mathfrak{X},2}t^2+\dots, \\
\phi'_t~~&=~~\phi+(F _{B ,1} \circ \phi+\phi_1-\phi \circ F _{A ,1})t+\phi'_{2}t^2+\dots\\
&=~~\phi+\phi'_{2}t^2+\dots.
\end{align*}
Furthermore, we may verify that $((\mu'_{A ,2},\mu'_{B ,2},\phi'_{2}),(P'_{A ,2},P'_{B ,2},0))$ is also a $2$-cocyle. Then, by repeating the argument,  it is equivalent to a trivial deformation. Thus, $\phi:(A,P_A) \to (B,P_B)$ is rigid.
\end{proof}

\section{$\mathrm{CCT}$ Theorem of Nijenhuis algebra morphisms}\label{sec:5}
Gerstenhaber and Schack \cite{Gerst1,Gerst2} give cohomology comparison Theorem ($\mathrm{CCT}$) to study a deformation theory of algebra
morphisms. In this section, we define the cohomology comparison Theorem ($\mathrm{CCT}$) of Nijenhuis algebras morphisms,
i.e., the cohomology of a Nijenhuis  algebra morphisms is isomorphic to the cohomology of
an auxiliary Nijenhuis algebra.
\begin{defn}\label{defi: morphism algebra}
Let $A ,B $ be two associative algebras, $\phi: A \rightarrow B $
be a morphism of associative algebras,
the mapping ring $\phi!$ is defined as $\phi!=A \oplus B \oplus B \phi$, the
multiplication is determined by associativity, the products in $A ,B $ and the conditions:
$x\cdot y=y\cdot x=\phi \cdot y=x\cdot \phi=\phi^2=0$ and $\phi \cdot x=\phi(x)\phi$, the unit of
$\phi!$ is $1_A +1_B $. In other words, the elements of $\phi!$ are of the form
$x+y_1+y_2\phi$ with $x\in A ,y_1,y_2\in B $ and
$(x+y_1+y_2\phi)(x'+y'_1+y'_2\phi)=xx'+y_1y'_1+(y_2\phi(x')+y_1y'_2)\phi$.
\end{defn}
\begin{defn}
Let $\psi:M\rightarrow N$ be a bimodule of $\phi: A \rightarrow B $, the
$\phi!$-module $\psi!$ is defined by
$\psi!=M\oplus N\oplus N \phi$, for $m+n_1+n_2\phi \in\psi!$,
$x+y_1+y_2\phi\in\phi!$, the
bimodule structure of $\phi!$ on
$\psi!$ is given by
\begin{align*}
(x+y_1+y_2\phi)(m+n_1+n_2\phi)&=xm+y_1n_1+(y_2\psi(m)+y_1n_2)\phi,\\
(m+n_1+n_2\phi)(x+y_1+y_2\phi)&=mx+n_1y_1+(n_2\phi(x)+n_1y_2)\phi.
\end{align*}
\end{defn}
As a consequence of the above definitions, we have the following result. 
\begin{lemma}\label{lem: phi RB algebra}
Let $\phi:(A,P_A )\rightarrow (B,P_B)$ be a morphism of Nijenhuis algebras. Define a linear map
$P_{\phi!}:\phi!\rightarrow \phi!$ by
$P_{\phi!}(x+y_1+y_2\phi)=P_A (x)+P_B (y_1)+P_B (y_2)\phi$,
where $\phi!=A \oplus B \oplus B \phi$ is an associative algebra by Definition
{\rm \ref{defi: morphism algebra}}.
Then $(\phi!,P_{\phi!})$ is a Nijenhuis algebra.
\end{lemma}
\begin{proof}
For $ x + y_1 + y_2\phi, x' + y_1' + y_2'\phi \in \phi! $, we have
\begin{align*}
&\quad P_{\phi!}(x + y_1 + y_2\phi)P_{\phi!}(x' + y_1' + y_2'\phi) \\
&= (P_A(x) + P_B(y_1) + P_B(y_2)\phi)(P_A(x') + P_B(y_1') + P_B(y_2')\phi) \\
& =P_A(x)P_A(x') + P_B(y_1)P_B(y_1') + (
P_B(y_2)\phi(P_A(x'))+
P_B(y_1)P_B(y_2'))\phi \\
&= P_A(x)P_A(x') + P_B(y_1)P_B(y_1') + (P_B(y_1)P_B(y_2') + P_B(y_2)P_B(\phi(x')))\phi \\
&= P_A\big(P_A(x)x' + xP_A(x') - P_A(xx')\big) + P_B\big(P_B(y_1)y_1' + y_1P_B(y_1') - P_B(y_1y_1')\big) \\
&\quad + \Big(P_B\big(P_B(y_1)y_2' + y_1P_B(y_2') - P_B(y_1y_2')\big) \\
&\quad + P_B\big(P_B(y_2)\phi(x') + y_2P_B(\phi(x')) - P_B(y_2\phi(x'))\big)\Big)\phi \\
&= P_{\phi!}\Big(P_A(x)x' + xP_A(x') - P_A(xx') + P_B(y_1)y_1' + y_1P_B(y_1') - P_B(y_1y_1') \\
&\quad + \big(P_B(y_1)y_2' + y_1P_B(y_2') - P_B(y_1y_2')\big)\phi \\
&\quad + \big(P_B(y_2)\phi(x') + y_2P_B(\phi(x')) - P_B(y_2\phi(x')) \big)\phi\Big) \\
&= P_{\phi!}\Big((P_A(x) + P_B(y_1) + P_B(y_2)\phi)(x' + y_1' + y_2'\phi) \\
&\quad + (x + y_1 + y_2\phi)(P_A(x') + P_B(y_1') + P_B(y_2')\phi) \\
&\quad - P_{\phi!}\big((x + y_1 + y_2\phi)(x' + y_1' + y_2'\phi)\big)\Big) \\
&= P_{\phi!}\Big(P_{\phi!}(x + y_1 + y_2\phi)(x' + y_1' + y_2'\phi) \\
&\quad + (x + y_1 + y_2\phi)P_{\phi!}(x' + y_1' + y_2'\phi) \\
&\quad - P_{\phi!}\big((x + y_1 + y_2\phi)(x' + y_1' + y_2'\phi)\big)\Big).
\end{align*}
Then, \( P_{\phi!} \) is a Nijenhuis operator. It follow that $(\phi!, P_{\phi!})$ is a Nijenhuis algebra.
\end{proof}
Let $\phi:(A,P_A )\rightarrow (B,P_B)$ be a morphism of Nijenhuis algebras,
$\langle (M,P_M ),(N,P_N ) ,\psi\rangle$ be a Nijenhuis $\phi$-bimodule. We define
$P_{\psi!}:\psi!\rightarrow \psi! $ by
$P_{\psi!}(m+n_1+n_2\phi)=P_M(m)+P_N(n_1)+P_N(n_2)\phi$,
in the same way, one may check that $(\psi!,P_{\psi!})$ is a Nijenhuis bimodule over
$(\phi!,P_{\phi!})$. 
By Remark \ref{remark morphism module},
 $\langle{}_{\rhd}M_\lhd,{}_\rhd N_\lhd,_\rhd\psi_\lhd\rangle$ 
 is a Nijenhuis
$\phi$-bimodule.
So  $\langle{}_{\rhd}\psi_\lhd!,P_{\psi!}\rangle$ is a bimodule over Nijenhuis algebra
$(\phi!,P_{\phi!})$.
\begin{defn}
Let $\phi:A \rightarrow B $ be an associative algebra morphism,
$\langle M,N ,\psi\rangle$ be a $\phi$-bimodule, define
$\tau^\bullet_\phi:C_{\mathrm{mor}}^\bullet(\phi,\psi)\rightarrow C_{\mathrm{Alg}}^\bullet(\phi!,\psi!)$ as follows:
for $\Gamma=(\Gamma^A ,\Gamma^B ,\Gamma^{A B })\in C_{\mathrm{mor}}^n(\phi,\psi)$,
$\tau^n_\phi \Gamma$ defined by
\begin{align*}
&\tau^n_\phi \Gamma|_{A ^{\otimes n}}=\Gamma^A ;
\tau^n_\phi \Gamma|_{B ^{\otimes n}}=\Gamma^B\\
&\text{for}~ ~(y\phi,x_2,\dots,x_n)\in B \phi\otimes A ^{n-1}\\
&\tau^n_\phi \Gamma(y\phi,x_2,\dots,x_n)=\Gamma^B (y,\phi(x_2),\dots,\phi(x_n))\phi +y \Gamma^{A B }(x_2,\dots,x_n)\phi \\
&\text{for}~ ~(y_1,\dots,y_{r-1},y_r\phi,x_{r+1},\dots,x_n)\in B ^{r-1}\otimes B \phi\otimes A ^{n-r}\\
&\tau^n_\phi \Gamma(y_1,\dots,y_{r-1},y_r\phi,x_{r+1},\dots,x_n)=
\Gamma^B (y_1,\dots,y_{r-1},y_r,\phi(x_{r+1}),\dots,\phi(x_n))\phi   \\
&\tau^n_\phi \Gamma(x_1,\dots,x_n)=0.~~~~~~~~Otherwise
\end{align*}
Then, $\tau^\bullet_\phi:C_{\mathrm{mor}}^\bullet(\phi,\psi)\rightarrow C_{\mathrm{Alg}} ^\bullet(\phi!,\psi!)$ is a
 quasi-isomorphism.
\end{defn}
Let $\phi: (A,P_A) \to (B,P_B) $ be a Nijenhuis algebra morphism and
$\langle (M,P_M ),(N,P_N ) ,\psi\rangle$ be 
a Nijenhuis $\phi$-bimodule, then
 $\langle{}_{\rhd}M_\lhd,{}_\rhd N_\lhd,_\rhd\psi_\lhd\rangle$ 
 is a Nijenhuis
$\phi$-bimodule.
Let $C^\bullet_{\mathrm{Alg}}(\phi!,{} _\rhd\psi_\lhd!)$ = $C^\bullet_{\mathrm{NjO}}(\phi!,\psi!)$.
We will show the following Lemma:
\begin{lemma}\label{lem: commu diagram}
The diagramm

\[
\xymatrix{
C_{\mathrm{mor}}^\bullet(\phi,\psi) \ar[d]^{\Theta^\bullet} \ar[r]^{\tau^\bullet_\phi} &C_{\mathrm{Alg}} ^\bullet
(\phi!,\psi!)
\ar[d]^{\Phi^\bullet}\\
C_{\mathrm{mor}}^\bullet(\phi,{}_\rhd\psi_\lhd) \ar[r]^{\tau^\bullet_{{}_\rhd\phi_\lhd}} &C^\bullet_{\mathrm{NjO}}
(\phi!,\psi!)}
\]
\medskip
is commutative, i.e., $\Phi^\bullet\circ\tau^\bullet_\phi
=\tau^\bullet_{{}_\rhd\phi_\lhd}\circ \Theta^\bullet.$
\end{lemma}

Now, we define 
$\tau^\bullet:C_{\mathrm{NjM}} ^\bullet(\phi,\psi)\rightarrow C^\bullet_{\mathrm{NjA}}
(\phi!,\psi!)$
to be 
$\tau^\bullet=\begin{pmatrix} \tau^{\bullet}_{\phi}&0\\0&\tau^{\bullet}_{{}_{\rhd}\phi_\lhd}
\end{pmatrix}$.
From Lemma \ref{lem: commu diagram}, $\tau^\bullet$ is a cochain map.
Thus, we get the following commutative diagram
\[\xymatrix{
		0\ar[r]& C_{\mathrm{mor}}^{\bullet-1}(\phi,_\rhd\psi_\lhd)\ar[r]\ar[d]^-{\tau^{\bullet-1}_{{}_\rhd\phi_\lhd}}&C_{\mathrm{NjM}}^\bullet(\phi,\psi)
\ar[r]\ar[d]^-{\tau^{\bullet}}&C_{\mathrm{mor}}^\bullet(\phi,\psi)\ar[r]\ar[d]^-{\tau^{\bullet}_{\phi}}&0\\
		0\ar[r]&C_{\mathrm{NjO}}^{\bullet-1}(\phi!,\psi!)\ar[r]& C^\bullet_{\mathrm{NjA}}(\phi!,\psi!)\ar[r]&  C_{\mathrm{Alg}}^\bullet(\phi!,\psi!)\ar[r]&0. }\]
Since $\mathrm{H}^\bullet(\tau^{\bullet}_{{}_\rhd\phi_\lhd})$ and $\mathrm{H}^\bullet(\tau^\bullet_\phi)$ are  isomorphisms, $\mathrm{H}^\bullet(\tau^\bullet)$ is also an isomorphism. It follows that the main Theorem. 
\begin{theorem}\label{thm: CCT of RB}
Suppose that $\phi:(A,P_A) \to (B,P_B)$ is a Nijenhuis 
algebras morphisms and $\langle (M,P_M),(N,P_N) ,\psi\rangle$ is a Nijenhuis $\phi$-bimodule. Let
$\mathrm{H}_{\mathrm{NjM}}^n(\phi,\psi)$ be the cohomology group of $\phi$ with coefficients in
$\langle (M,P_M),(N,P_N) ,\psi\rangle$ and
$\mathrm{H}_{\mathrm{NjA}}^n(\phi!,\psi!)$ be the cohomology group of Nijenhuis 
algebra $(\phi!,P_{\phi!})$ with coefficients in $(\psi!,P_{\psi!})$. Then
$\mathrm{H}_{\mathrm{NjM}}^n(\phi,\psi)\cong \mathrm{H}_{\mathrm{NjA}}^n(\phi!,\psi!)$.
\end{theorem}

\section{Minimal model of operad for Nijenhuis algebras morphisms}\label{sec:6}
Since operad of Nijenhuis algebras is Koszul, 
so we can construct minimal models of Nijnhuis algebra morphisms operads by \cite{Do}, 
i.e., we will construct $2$-colored operads over Nijenhuis algebras and their minimal models.
\begin{defn}(\cite{Song}) 
The \textit{operads} for Nijenhuis algebras is defined to be the quotient of
the \textit{free graded operad} $F(M)$ generated by a \textit{graded collection} $M$ by an \textit{operadic ideal} $I$, where
the collection $M$ is given by $M(1) = \bk P, M(2) = \bk \mu$ and $M(n) = 0$ for $n\neq 1, 2$ and $I$ is
generated by
\begin{align*}
&\mu\circ_1 \mu-\mu\circ_2\mu, \\
&(\mu \circ_1 P)\circ_2 P -(P \circ_1 \mu)\circ_1 P-(P \circ_1 \mu)\circ_2 P + (P \circ_{1} P) \circ_{1} \mu.
\end{align*}
Denote this operad by $\mathfrak{RjU}$.
\end{defn}
Then minimal model of  $\mathfrak{RjU}$ may be defined in \cite{Song} as follows.
\begin{defn}\label{def-minimodel1}(\cite{Song}) 
Let $\mathcal{O} = (\mathcal{O}(1),\dots, \mathcal{O}(n), \dots)$ be the \textit{graded collection} where $\mathcal{O}(1) = \bk P_1$ with
$|P_1|=0$ and for $n \geqslant 2$, $\mathcal{O}(n) = \bk P_n \oplus \bk m_n$ with $|P_n| = n-1, |m_n|= n-2$. The \textit{operad $\mathfrak{RjU}_{\infty}$
of homotopy Nijenhuis algebras} is defined by the \textit{differential graded operad} $(\mathcal{F(O)}, \partial)$, where
the underlying \textit{free graded operad} is generated by the graded collection $\mathcal{O}$ and the \textit{action of the
diﬀerential} $\partial$ on generators is given by the following equations. For each $n\geqslant 2$,
\begin{equation*} \label{Eq. homotopy NjA eq1}
	\begin{aligned}
		\partial(m_n)
		 =&\ \sum_{j=2}^{n-1}\sum_{i=1}^{n-j+1}(-1)^{i+j(n-i)}m_{n-j+1}\circ_i m_j,
	\end{aligned}
\end{equation*}
and for $ n \geqslant 1 $,
\begin{equation*} \label{Eq. homotopy NjA eq2}
	\begin{aligned}
			 \partial(P_{n})
			= &\ \sum_{ \substack{ r_1+\dots+r_p =n
					\\	r_1, \dots, r_p \geqslant 1
					\\	2 \leqslant p \leqslant n} }
			\sum_{ t=0}^{p}
			\sum_{\substack{ 1 \leqslant i_{1} \leqslant r_{1} \\ \dots \\ 1 \leqslant i_{t} \leqslant r_{t} }}
			\sum_{1 \leqslant k_{t} < \dots < k_{p-1} \leqslant p}
			(-1)^{\alpha'} \\
			& P_{r_{1}} \circ_{i_{1}}
			\Big(
			P_{r_{2}} \circ_{i_{2}}
			\big(
			\dots
			\circ_{i_{t-1}}
			\Big(
			P_{r_{t}} \circ_{i_{t}}
			\big(
			(
			\dots
			((m_{p} \circ_{k_{t}} P_{r_{t+1}}) \circ_{k_{t+1}+r_{t+1}-1} P_{r_{t+2}} )
			\dots
			) \circ_{\beta} P_{r_{p}}
			\big)
			\Big)
			\dots
			\big)
			\Big),
	\end{aligned}
\end{equation*}
where	
\begin{align*}
	\beta &= k_{p-1}+r_{t+1} + \dots + r_{p-1}-(p-1-t),\\
	\alpha'
	&= 1+ \sum^{t}_{q=1} \Big(i_{q} + \big(\sum_{s=q+1}^{p} r_{s}\big) (r_{q}-i_{q}) - q(r_{q}-i_{q})\Big)
	+ \sum_{i=t+1}^{p}(k_{i-1} -p)(r_{i}-1),
\end{align*}
and  $\mathfrak{RjU}_{\infty}$ is the \textit{minimal model of the operad $\mathfrak{RjU}$ for Nijenhuis algebras}.
\end{defn}
Similar to \cite[Example 4]{Mark4}, we have the following definition.
\begin{defn}
Let $\mathfrak{Rj}\mathfrak{U}$ be the operad for Nijenhuis algebras. Then there is a \textit{$2$-colored operad} $\mathfrak{Rj}\mathfrak{U}^{\mathbb{A}\rightarrow \mathbb{B}}$  whose algebras are of the
form $f : A \rightarrow B$ , in which $A$ and $B$ are \textit{$\mathfrak{Rj}\mathfrak{U}$-algebras} and $f$ is a morphism of $\mathfrak{Rj}\mathfrak{U}$-algebras, which is constructed as the following
\begin{align*}
\mathfrak{RjU}^{\mathbb{A}\rightarrow \mathbb{B}}=\frac{\mathfrak{RjU}^A*\mathfrak{RjU}^B*F(f)}{(fx_A=x_Bf^\otimes, x\in \mathfrak{RjU}(n))},
\end{align*}
where $\mathfrak{Rj}\mathfrak{U}^A$ and $\mathfrak{Rj}\mathfrak{U}^B$ are copies of $\mathfrak{Rj}\mathfrak{U}$  concentrated in the \textit{colors} $\mathbb{A}$ and $\mathbb{B}$, respectively,
$x_A$ and $x_B$ are the respective copies of $x$ in $\mathfrak{Rj}\mathfrak{U}^A$ and $\mathfrak{Rj}\mathfrak{U}^B$, and $F(f)$ is the \textit{
free $2$-colored operad} on the generator $f : A\rightarrow B$. The star $*$ denotes the \textit{free product (= the coproduct) of $2$-colored operad.} 

Precisely, $\mathfrak{Rj}\mathfrak{U}^{\mathbb{A}\rightarrow \mathbb{B}}$ is defined to be the quotient of
the \textit{free graded operad} $F(N)$ generated by a \textit{graded collection} $N$ by an operadic ideal $I$, where
the \textit{collection} $N$ is given by $N(1) = \bk\{P_A,P_B,f\}, N(2) = \bk\{\mu_A,\mu_B\}$ and $N(n) = 0$ for $n\neq 1, 2$ and $I$ is
generated by
\begin{align*}
&\mu_{X}\circ_1 \mu_{X}-\mu_{X}\circ_2\mu_{X}, \\
&(\mu_X \circ_1 P_X)\circ_2 P_X -(P_X \circ_1 \mu_X)\circ_1 P_X-(P_X \circ_1 \mu_X)\circ_2 P_X + (P_X \circ_{1} P_X) \circ_{1} \mu_X,\\
&f\circ_1\mu_A-(\mu_B\circ_1 f)\circ_2f,\\
&f\circ_1 P_A-P_B\circ_1 f,
\end{align*}
where $X=\{A,B\}$.
\end{defn}

Since the operad $\mathfrak{Rj}\mathfrak{U}$ is Koszul, so we can construct the minimal model of $\mathfrak{Rj}\mathfrak{U}^{\mathbb{A}\rightarrow \mathbb{B}}$ by Dotsenko-Poncin \cite{Do}. Recall homotopy cooperad $\mathfrak{Rj}\mathfrak{U}^{\text{!`}}$ in \cite{Song} as follows.

The \textit{homotopy cooperad} $\mathscr{S}(\mathfrak{Rj}\mathfrak{U}^\text{!`})\ot \mathcal{S}^{-1}$, the \textit{Hadamard product} of $\mathscr{S}(\mathfrak{Rj}\mathfrak{U}^\text{!`})$ and $\mathcal{S}^{-1}$, is called the \textit{Koszul dual homotopy cooperad} of $\mathfrak{Rj}\mathfrak{U}$, denoted by $\mathfrak{Rj}\mathfrak{U}^{\text{!`}}$
.

Precisely, the underlying \textit{graded collection} of ${\mathfrak{Rj}\mathfrak{U}^\text{!`}}$ is
$${\mathfrak{Rj}\mathfrak{U}^\text{!`}}(n)=\bk e_n\oplus \bk o_n$$
with $e_n=u_n\ot \varepsilon_n$ and $o_n=v_n\ot \varepsilon_n$ for $n\geqslant 1$, thus $|e_n|=n-1$ and $|o_n|=n$. Then the \textit{coaugmented homotopy cooperad structure on the graded collection} $\mathscr{S}(\mathfrak{RjU}^{\text{!`}})=\bk \widetilde{m}_n+\bk \widetilde{P}_n$ is defined as \cite{Song}.

we denote $\mathfrak{RjU}^{\mathbb{A}\rightarrow \mathbb{B}}$ by $\mathfrak{RjU}_{\bullet\to\bullet}$ and $\mathscr{S}(\mathfrak{RjU}^{\text{!`}})\triangleq \mathfrak{RjU}^{\text{!`}}$. Let us consider $\{A,B\}$-\textit{colored $\mathbb{S}$-module} 
$$\mathcal{M}_{\bullet\to \bullet}=\mathfrak{\overline{RjU}}^{\text{!`}}_{A\to A}\oplus \mathfrak{\overline{RjU}}^{\text{!`}}_{B\to B}\oplus s \mathfrak{RjU}^{\text{!`}}_{A\to B}.$$

We define the \textit{cobar complex} $\Omega(\mathcal{M}_{\bullet\to \bullet})$ by $\mathfrak{RjU}_{\bullet\to \bullet,\infty}.$ By \cite{Do}, it follows that 

\begin{prop}\label{min-model} 
$\mathfrak{RjU}_{\bullet\to \bullet,\infty}$ is the \textit{minimal model} of $\mathfrak{RjU}_{\bullet\to\bullet}$.
\end{prop}
The general construction  produces an \textit{$L_\infty$-algebra} structure on the \textit{space of $\mathbb{S}$-module morphisms}
$$L_{A,B}=\Hom (\mathcal{M}_{\bullet\to \bullet}, \mathrm{End}_{A\oplus B}),$$
where $\mathrm{End}_{A\oplus B}$ is $\{A,B\}$-\textit{colored operad}. This space of morphisms can be naturally identified with the space
$$(f_a,f_b, f_{ab})\in \Hom(\overline{\mathfrak{RjU}}^{\text{!`}}(A), A)\oplus\Hom (\overline{\mathfrak{RjU}}^{\text{!`}}(B), B)\oplus\Hom (s\mathfrak{RjU}^{\text{!`}}(A), B).$$  By means of \cite{Me} for \textit{properads}, $(f_a,f_b, f_{ab})$ is a solution to the \textit{Maurer–Cartan equation} of the $L_\infty$-algebra $L_{A,B}$ if and only if $f_a$ is a structure of a \textit{homotopy $\mathfrak{RjU}$-algebra} on $A$, $f_b$ is a structure of a \textit{homotopy $\mathfrak{RjU}$-algebra} on $B$, and $f_{ab}$ is a \textit{homotopy morphism} between these algebras.

\end{document}